\documentclass[11pt,a4paper]{article}
\usepackage{amsmath,amsthm,amssymb}

\newtheorem{theorem}{Theorem}
\newtheorem{proposition}[theorem]{Proposition}
\newtheorem{lemma}[theorem]{Lemma}

\newtheorem{Conjecture}[theorem]{Conjecture}

\title{{\bf Ideal class groups of some quadratic number fields 
and factorization of values of some quadratic polynomials}}
\author{
St\'ephane R. LOUBOUTIN\\
Aix Marseille Universit\'e, CNRS, Centrale Marseille, I2M,\\ 
Marseille, France\\
stephane.louboutin@univ-amu.fr}
\date{\today}
\begin{document}
\bibliographystyle{alpha}
\maketitle
\footnotetext{

Imaginary quadratic field. 
Class group. 
Class number. 
Quadratic polynomial. 
 Frobenius-Rabinowitsch}

\begin{abstract}
We fill the gaps in A. Gica's determination of all the odd positive integers $d$ 
for which the number of distinct prime divisors of $f_d(x)=d+x^2$ is less than or equal to $2$ 
for all the positive and odd integers $x\leq\sqrt{d}$. 
We also determine all the even positive integers $d$ 
for which the number of distinct prime divisors of $f_d(x)$ is less than or equal to $2$ 
for all the positive and even integers $x\leq\sqrt{d}$. 
These problems are related to the famous Frobenius-Rabinowitsch's characterization 
of the imaginary quadratic number fields ${\mathbb Q}(\sqrt{-d})$ of odd discriminants with class number one 
in terms of the primality of $f_d(x)/4$ for all the positive and odd integers $x\leq\sqrt{d}$. 
However, the solution to our problem is much more difficult to come up with. 
We also begin to address the same problems for the case of $f_d(x)=d-x^2$, 
in relation with the class groups of the real quadratic number fields ${\mathbb Q}(\sqrt{d})$.
\end{abstract}

\section{Introduction}
When we began to read carefully the article \cite{GicaIndMath},
 we quickly realized that the author's approach did not entirely satisfy us, 
 but above all, that there were gaps in some of his proofs. 
 (See Section \ref{remarks} for details.)
 In the present paper we fill these gaps 
 using an approach based on algebraic number theory and extend A. Gica's results, 
 e.g. see Theorems \ref{mainthd=2mod4}, \ref{mainthd=2mod4notsquarefree} and \ref{mainthd=2mod4reel}.
 
For $0<d\equiv 3\pmod 4$, Frobenius-Rabinowitsch's Theorem asserts that 
$(d+x^2)/4$ is prime or equal to $1$ for all the odd integers $x$ in the range $1\leq x\leq\sqrt{d}$ 
if and only if $d$ is prime 
and the class number of the imaginary quadratic number field ${\mathbb Q}(\sqrt{-d})$ 
of discriminant $d_{\mathbb K}=-d$ 
is equal to one 
(see \cite{Ribemboin}, or \cite[Th\'eor\`eme 2]{LouCRAS312} for generalizations). 
According to A. Baker and H. M. Stark's solution to the class number one problem 
for the imaginary quadratic fields, 
e.g. see \cite{Goldfeld},
this happens if and only if $d\in\{3,7,11,19,43,67,163\}$ ($7$ values).

In the same way, using \cite{LouCRAS312} it is easy to see that for $0<d\equiv 2\pmod 4$ the $(d+x^2)/2$'s 
are prime or equal to $1$ for all the even integers $x$ in the range $0\leq x\leq\sqrt{d}$ 
if and only if $d/2$ is prime 
and the class number of the imaginary quadratic number field ${\mathbb Q}(\sqrt{-d})$ 
of discriminant $d_{\mathbb K}=-4d$ 
is equal to two. 
According to \cite{Watkins}, 
this happens if and only if $d\in\{2,6,10,22,58\}$ ($5$ values). 

Our aim here is to tackle a more difficult problem by alleviating the constraint on the prime factorizations 
of the $d+x^2$'s. 
Let $\omega(n)$ denote the number of distinct prime divisors of an integer $n\geq 2$. 
Set
$$M_{odd}(d)
:=\max\{\omega(d+x^2);\text{ $1\leq x\leq\sqrt d$ and $x$ odd}\}.$$
{\bf First}, using only elementary tools, 
we will prove the following result that corrects \cite[Theorem 7.1]{GicaIndMath}, 
where the values $1$ and $135$ are missing:

\noindent\frame{\vbox{
\begin{theorem}\label{dnotprimenotpq}
Let $d\geq 1$ be odd, neither prime nor the product of two distinct primes. 
Then $M_{odd}(d)\leq 2$ if and only if $d\in\{1, 9, 25, 27, 49, 63, 135, 175, 207, 343\}$ 
($10$ values).
\end{theorem}
}}

{\bf Second}, using algebraic number theory 
and the determination of all the imaginary quadratic number fields with small class numbers, 
we will consider the case $d\not\equiv 7\pmod 8$ an odd positive integer 
that is prime or the product of two distinct primes:

\noindent\frame{\vbox{
\begin{theorem}\label{mainthimpairnot7mod8}
Let $d\geq 1$ with $d\not\equiv 7\pmod 8$ be an odd positive integer 
that is prime or the product of two distinct primes. 
If $M_{odd}(d)\leq 2$ 
then $d$ is prime and the class number of the imaginary quadratic field ${\mathbb Q}(\sqrt{-d})$ divides $4$.
Then, $M_{odd}(d)\leq 2$ if and only if $d$ is prime and
$d\in\{3,$ $5,$ $11,$ $13,$ $17,$ $19,$ $37,$ $43,$ $67,$ $73,$ $97,$ 
$163,$ $193\}$ 
($13$ values). 
\end{theorem}
}}

{\bf Third}, using algebraic number theory and the Siegel-Tatuzawa theorem 
to justify that we are missing at most one value of $d$, 
we will prove the complementary case of Theorem \ref{dnotprimenotpq}:

\noindent\frame{\vbox{
\begin{theorem}\label{mainthimpair7mod8}
Let $d\geq 1$ with $d\equiv 7\pmod 8$ be an odd positive integer 
that is prime or the product of two distinct primes. 
If $M_{odd}(d)\leq 2$ 
then the ideal class group of the imaginary quadratic field ${\mathbb K} ={\mathbb Q}(\sqrt{-d})$ 
is cyclic generated by the ideal class of any of the two prime ideals of ${\mathbb K}$ above the prime $2$ 
and the class number of ${\mathbb K}$ is not large, namely less than or equal to $(\log 2d)/\log 2$.
Then, with at most one possible exception 
we have $M_{odd}(d)\leq 2$ if and only if 
$d$ is prime and 
$d\in\{7,$ $23,$ $31,$ $47,$ $79,$ $103,$ $127,$ $151,$ $223,$ $463,$ $487,$ $823,$ $1087,$ $1423\}$ 
($14$ values), 
or $d$ is the product of two distinct primes
$d\in\{15,$ $39,$ $55,$ $247,$ $583\}$ 
($5$ values). 
Assuming the Restricted Riemann Hypothesis $\zeta_{\mathbb K}(1-(2/\log d_{\mathbb K}))\leq 0$ 
for the Dedekind zeta function $\zeta_{\mathbb K}(s)$ of all imaginary quadratic fields ${\mathbb K}$, 
there is no such exception.
\end{theorem}
}}

{\bf Fourth}, After exchanging several letters with A. Gica, 
it appears that by not giving the complete proof of \cite[Theorem 7.2]{GicaIndMath}, 
its statement and the outline of its proof are incorrect.
He wanted to prove the following Theorem \ref{mainthd=2p}, 
for which we will provide a different proof than the one he sent us on November 2025. 

\begin{theorem}\label{mainthd=2p}
Let $d> 2$ with $d\equiv 2\pmod 4$ be a square-free integer. 
Assume that $\omega(d+x^2)\leq 2$ for all the integers $x$ in the range $0\leq x\leq\sqrt{d}$. 
Then the ideal class group of the imaginary quadratic number field ${\mathbb Q}(\sqrt{-d})$ 
is cyclic of order dividing $4$. 
Therefore, $\omega(d+x^2)\leq 2$ for all the integers $x$ in the range $0\leq x\leq\sqrt{d}$ 
if and only if $d\in\{2,$ $6,$ $10,$ $14,$ $22,$ $34,$ $46,$ $58,$ $82,$ $142\}$ 
($10$ values).
\end{theorem}

Now, in Theorem \ref{mainthd=2mod4}, we prove a result 
stronger than \cite[Theorem 7.2]{GicaIndMath} A. Gica intended to prove.
Setting
$$M_{even}(d)
:=\max\{\omega(d+x^2);\text{ $2\leq x\leq\sqrt d$ and $x$ even}\},$$
we prove the following result:

\noindent\frame{\vbox{
\begin{theorem}\label{mainthd=2mod4}
Let $d\equiv 2\pmod 4$ be a positive square-free integer. 
Set ${\mathbb K}={\mathbb Q}(\sqrt{-d})$, an imaginary quadratic number field 
of discriminant $d_{\mathbb K} =-4d$ and ring of algebraic integers ${\mathbb Z}[\sqrt{-d}]$.
Assume that $M_{even}(d)\leq 2$. 
Then the class number of ${\mathbb K}$ divides $16$. 
This happens if and only if 
$d\in\{2,$ $6,$ $10,$ $14,$ $22,$ $30,$ $34,$ $46,$ $58,$ $70,$ $82,$ $142\}$ 
($12$ values). 
\end{theorem}
}}

\noindent\frame{\vbox{
\begin{theorem}\label{mainthd=2mod4notsquarefree}
Let $d\equiv 2\pmod 4$ be a positive not square-free integer. 
Then $M_{even}(d)\leq 2$ if and only if $d=18$.
\end{theorem}
}}

In contrast with Theorem \ref{mainthd=2mod4}, 
we point out that the determination of all the square-free positive integers $d\equiv 2\pmod 4$ 
for which $M_{odd}(d)\leq 2$ is not yet done. 
According to A. Gica in \cite{GicaIndMath} there are conjecturally $202$ such $d$'s: 

\begin{Conjecture}Let $d\equiv 2\pmod 4$ be a positive square-free integer. 
Then, $M_{odd}(d)\leq 2$ if and only if 
$d\in\{2,$ $6,$ $10,$ $14,$ $22,$ $26,$ $30,$ $34,$ $38,$ $42,$ 
$46,$ $58,$ $62,$ $66,$ $70,$ $74,$ $78,$ $82,$ $86,$ $94,$ 
$102,$ $106,$ $110,$ $118,$ $122,$ $130,$ $138,$ $142,$ 154,$ 158,$ 
$166,$ $178,$ $190,$ $202,$ $210,$ $214,$ $218,$ $226,$ $238,$ $262,$ 
$274,$ $282,$ $298,$ $302,$ $310,$ $322,$ $346,$ $358,$ $366,$ $382,$ 
$394,$ $418,$ $422,$ $442,$ $466,$ $478,$ $498,$ $502,$ $518,$ $526,$ 
$538,$ $562,$ $598,$ $610,$ $622,$ $658,$ $682,$ $694,$ $718,$ $730,$ 
$742,$ $754,$ $778,$ $802,$ $826,$ $858,$ $862,$ $898,$ $958,$ $982,$ 
$1030,$ $1090,$ $1138,$ $1162,$ $1198,$ $1222,$ $1282,$ $1318,$ $1366,$ $1402,$ 
$1558,$ $1582,$ $1618,$ $1642,$ $1738,$ $1822,$ $1870,$ $1918,$ $1978,$ $2002,$ 
$2038,$ $2062,$ $2158,$ $2182,$ $2242,$ $2302,$ $2398,$ $2458,$ $2482,$ $2542,$ 
$2578,$ $2818,$ $2878,$ $2902,$ $2962,$ $2998,$ $3298,$ $3322,$ $3382,$ $3502,$ 
$3658,$ $3802,$ $3958,$ $4162,$ $4222,$ $4258,$ $4558,$ $4678,$ $4918,$ $5098,$ 
$5182,$ $5338,$ $5602,$ $5758,$ $5842,$ $6238,$ $6262,$ $6598,$ $6658,$ $6742,$ 
$6862,$ $7078,$ $7282,$ $7522,$ $8002,$ $8338,$ $8782,$ $9262,$ $9718,$ $10138,$ 
$10822,$ $10858,$ $11278,$ $11302,$ $12142,$ $12202,$ $12538,$ $12742,$ $13798,$ $13918,$ 
$14422,$ $14722,$ $15082,$ $15178,$ $16102,$ $17158,$ $18202,$ $18418,$ $19462,$ $21058,$ 
$23398,$ $23662,$ $24082,$ $25162,$ $25642,$ $26398,$ $27358,$ $28582,$ $29362,$ $30178,$ 
$30622,$ $31882,$ $32362,$ $33742,$ $34318,$ $35722,$ $38578,$ $41218,$ $45742,$ $47338,$ 
$48742,$ $61462,$ $62302,$ $83218,$ $85402,$ $92698,$ $92878,$ $94378,$ $102958,$ $166798,$ 
$225142,$ $288502\}$ 
($202$ values).
\end{Conjecture}

Bearing on extended numerical computations, we also conjecture the following:

\begin{Conjecture}
Let $d\equiv 2\pmod 4$ be a positive not square-free integer. 
Then, $M_{odd}(d)\leq 2$ if and only if 
$d\in\{18,$ $50,$ $54,$ $90,$ $98,$ $126,$ $162,$ $198,$ 
$242,$ $250,$ $294,$ $342,$ $378,$ $450,$ $522,$ $550,$ $558,$ $702,$ 
$722,$ $850,$ $882,$ $918,$ $1078,$ $1150,$ $1422,$ $1450,$ $2662,$ $2842,$ 
$3250,$ $3798,$ $4018,$ $4698,$ $4750,$ $5350,$ $7018,$ $9802,$ $11650,$ $12838,$ $16762,$ $17182,$ 
$20938,$ $23998,$ $30682,$ $48778,\}$ 
($44$ values).
\end{Conjecture}

{\bf Fifth} and finally, to consider the same problems but for real quadratic number fields, we set 
$$M_{even}'(d)
:=\max\{\omega(d-x^2);\text{ $2\leq x\leq\sqrt d$ and $x$ even}\}.$$ 
For $0<d\equiv 5\pmod 8$, Frobenius-Rabinowitsch's Theorem for real quadratic fields 
asserts that 
$(d-x^2)/4$ is prime or equal to $1$ for all the odd integers $x$ in the range $3\leq x\leq\sqrt{d}$ 
if and only if $d$ is square-free of the form $d=m^2\pm 4$ or $4m^2+1$ 
and the class number $h_d$ of the real quadratic number field ${\mathbb Q}(\sqrt{d})$ 
is equal to one 
(see \cite[Theorem 2]{LouCJM42}. 
The fundamental units $\epsilon_d =(m+\sqrt{d})/2$ or $2m+\sqrt{d}$ 
of these families of real quadratic fields 
are small, 
namely $\log\epsilon_d \ll\log d$. 
It follows from the Brauer-Siegel theorem that there are only finitely many such $d$'s 
with $h_d =1$. 
The only known such $d$'s were 
$d=13$, $21$, $29$, $37$, $53$, $77$, $101$, $173$, $197$, $293$, $437$ and $677$ 
($12$ values).
Following the seminal work of A. Bir\'o in \cite{Biro1} for the case $d=m^2+4$, 
it follows from \cite{Biro2}, \cite{BKL} and \cite{BG} that this list is complete.
With a proof similar to that of Theorem \ref{mainthd=2mod4}, 
we have:

\noindent\frame{\vbox{
\begin{theorem}\label{mainthd=2mod4reel}
Let $d\equiv 2\pmod 4$ be a positive square-free integer. 
Set ${\mathbb K}={\mathbb Q}(\sqrt{d})$, a real quadratic number field 
of discriminant $d_{\mathbb K} =4d$ and ring of algebraic integers ${\mathbb Z}[\sqrt{d}]$.
Assume that $M_{even}'(d)\leq 2$. 
Then the class number of ${\mathbb K}$ divides $16$. 
\end{theorem}
}}

We could not prove that $M_{even}'(d)\leq 2$ implies that the fundamental unit $\epsilon_d>1$ 
is small.
Hence we cannot use the Brauer-Siegel theorem to deduce that there are only finitely many such $d$'s. 
However, easy numerical computations suggest that we have:

\noindent\frame{\vbox{
\begin{Conjecture}
Let $d\equiv 2\pmod 4$ be a positive square-free integer. 
Then $M_{even}'(d)\leq 2$ even if and only if 
$d\in\{2,$ $6,$ $10,$ $14,$ $22,$ $26$, $30,$ $38,$ $42$, $62,$ $110,$ $122,$ $182,$ $278$, $362$, $398\}$ 
($16$ values). 
\end{Conjecture}
}}

At least, we will solve this problem for the case of not square-free integers:

\noindent\frame{\vbox{
\begin{theorem}\label{mainthd=2mod4notsquarefreereel}
Let $d\equiv 2\pmod 4$ be a positive not square-free integer. 
Then $M_{even}'(d)\leq 2$ if and only if 
$d\in\{18,$ $50,$ $54,$ $90$, $98\}$ 
($5$ values). 
\end{theorem}
}}

\section{Proof of Theorem \ref{dnotprimenotpq}}
Theorem \ref{dnotprimenotpq} follows from the following Lemmas \ref{step3}, \ref{step4} and \ref{step5}.

\begin{lemma}\label{step1}
Let $d>1$ be odd. 
Let $p$ be a prime dividing $d$ such that $9p^2\leq d$. 
Assume 
$(i)$ that $p^2$ does not divide $d$, 
or $(ii)$ that $p>3$ and $p^3$ divides $d$ 
or $(iii)$ that $p=3$ and $p^5$ divides $d$. 
Then $M_{odd}(d)\geq 3$.
\end{lemma}

\begin{proof}
Suppose that $M_{odd}(d)\leq 2$ and let us arrive at some contradiction.\\
In case $(i)$, we have $d+p^2=pd'$ and $d+(3p)^2=pd''$, where $d'$ and $d''$ are even and not divisible by $p$. 
Hence $\omega(d+p^2)\leq M_{odd}(d)\leq 2$ and $\omega(d+(3p)^2)\leq M_{odd}(d)\leq 2$ 
give $d+p^2=2^\alpha p$ and $d+(3p)^2=2^\beta p$. 
Therefore $8p=2^\alpha(2^{\beta-\alpha}-1)$, by subtraction, which implies $\alpha =3$ 
and the contradiction $10p^2\leq d+p^2=2^\alpha p=8p$.\\
In case $(ii)$ the same reasoning gives $d+p^2=2^\alpha p^2$ and $d+(3p)^2=2^\beta p^2$. 
Therefore $8=2^\alpha(2^{\beta-\alpha}-1)$, by subtraction, which implies $\alpha =3$ 
and the contradiction $10p^2\leq d+p^2=2^\alpha p^2=8p^2$.\\
In case $(iii)$, since $2\mid d+3^2$, $3^2\parallel d+3^2$ and $\omega(d+3^2)\leq M_{odd}(d)\leq 2$ 
we have $d+3^2=2^\alpha\cdot 3^2$. Using the same arguments, we have $d+3^4=2^\beta\cdot 3^4$. 
Therefore, $8d =3^2(d+3^2)-(d+3^4) =2^\beta(2^{\alpha-\beta}-1)\cdot 3^4$. 
Hence $\beta=3$, $d+3^4=2^\beta\cdot 3^4=2^3\cdot 3^4$, $d=7\cdot 3^4$ 
and we get the contradiction that $d =7\cdot 3^4$ would be divisible by $3^5$.
\end{proof}

\begin{lemma}\label{step2}
Let $d>1$ be odd and not prime. 
If $\omega(d)\geq 3$, then $M_{odd}(d)\geq 3$.
\end{lemma}

\begin{proof}
Write $d=\prod_{i=1}^rp_i^{e_i}$, where $r:=\omega(d)\geq 3$.\\ 
On the one hand, suppose that one of the $e_i$'s, say $e_1$, satisfies $e_1\geq 2$. 
We may assume that $3\leq p_2<\cdots <p_r$. 
Then, $d\geq p_1^{e_1}p_2^{r-1}\geq (p_1p_2)^2$. 
Since $2$, $p_1$ and $p_2$ divide $d+(p_1p_2)^2$, 
we do get $M_{odd}(d)\geq \omega(d+(p_1p_2)^2)\geq 3$.\\
On the other hand, suppose that $e_1=\cdots=e_r=1$.
We may assume that $3\leq p_1<\cdots <p_r$. 
Noticing that $t(t+2)(t+4)-9t^2 =t(t^2-3t+8)>0$ for $t>0$, we get $d\geq p_1(p_1+2)(p_1+4)>9p_1^2$. 
Applying Lemma \ref{step1}$(i)$ we also do get $M_{odd}(d)\geq 3$.
\end{proof}

\begin{lemma}\label{step3}
Let $d>1$ be odd and not prime. 
Assume that $M_{odd}(d)\leq 2$. 
Then, $d=3^3=27$, $d=7^3=343$, $d=3^3\cdot 5=135$ 
(this value is missing in \cite[Theorem 7.1]{GicaIndMath}), 
$d=p^2$ with $p$ prime, 
or $d=p^aq$ with $p\neq q$ prime and $a\in\{1,2\}$.
\end{lemma}

\begin{proof}
By Lemma \ref{step2}, we have $\omega(d)\leq 2$.\\
{\bf First}, suppose that $\omega(d)=1$. Hence, $d=p^a$ for some $a\geq 2$.
If $p>3$ and $a\geq 3$ then $9p^2>d=p^a$, by Lemma \ref{step1}$(ii)$. 
Hence, $p^{a-2}<9$, i.e. $p^{a-2}=5$ or $7$ and $d=p^a=5^3=125$ or $d=7^3=343$. 
Since $\omega(5^3+1)=\omega(126)=3$, the case $d=5^3$ is excluded. 
If $p=3$ then $d=3^a$ with $2\leq a\leq 4$, by Lemma \ref{step1}$(iii)$. 
Since $\omega(3^4+3^2)=\omega(90)=3$, we get $d=3^2=9$ or $d=3^3=27$.\\
{\bf Second}, suppose that $\omega(d)=2$. Hence, $d=p^aq^b$ with $a,b\geq 1$. 
If we had $a,b\geq 2$ then $2$, $p$ and $q$ would divide $d+(pq)^2$ 
and we would have $M_{odd}(d)\geq \omega(d+(pq)^2)\geq 3$.\
Therefore, $d=p^aq$ with $a\geq 1$.\\
If $p>3$ and $a\geq 3$, 
then $d\geq pqp^2\geq 15p^2\geq 9p^2$ and $a\in\{1,2\}$, by Lemma \ref{step1}$(ii)$.\\ 
Now, assume that $p=3$ and $a\geq 3$, i.e. that $d=3^aq$ with $a\geq 3$ and $q\geq 5$. 
If we had $d\ge 15^2$ then we would have $\omega(d+3^2)\leq 2$ and $\omega(d+15^2)\leq 2$, 
which would yield $3^{a-2}q+1=2^\alpha$ and $3^{a-2}q+25=2\beta$, 
hence $24=2^\alpha(2^{\beta-\alpha-1})$ and the contradiction $8=2^\alpha =3^{a-2}q+1\geq 3q+1\geq 16$. 
Hence $d=3^aq<15^2$ with $a\geq 3$, which gives $d=3^3\cdot 5=135$ or $d=3^3\cdot 7=189$. 
Since $\omega(189+1^2)=\omega(190)=3$ we have finished the proof.
\end{proof}

\begin{lemma}\label{step4}
Assume that $d=p^2$ with $p\geq 3$ prime. 
Then $M_{odd}(d)\leq 2$ if and only if $d\in\{3^2=9, 5^2=25, 7^2=49\}$.
\end{lemma}

\begin{proof}
Let $p\geq 3$ be an odd prime. 
Take $x_p\in (-65,65]\cap {\mathbb Z}$ such that $x_p\equiv 47p\pmod{130}$. 
Then $x_p$ is odd and $p^2+\vert x_p\vert^2\equiv p^2(1+47^2)\equiv 0\pmod{130}$. 
Hence $M_{odd}(p^2)\geq\omega(p^2+x_p^2)\geq\omega(130)= 3$ for $p\geq 65$. 
If $11\leq p\leq 61$ then either $p^2\equiv 1\pmod{10}$, in which case we set $y_p=3$, 
or $p^2\equiv 9\pmod{10}$, in which case we set $y_p=1$.
Then $p^2+y_p^2\equiv 0\pmod{10}$, hence $M_{odd}(p^2)\geq\omega (p^2+y_p^2)\geq 2$. 
Now numerical verifications show that in fact $\omega (p^2+y_p^2)\geq 3$ for $11\leq p\leq 65$. 
Hence $M_{odd}(p^2)\geq3$ for $11\leq p\leq 65$. 
Finally, $M_{odd}(3^2) =M_{odd}(5^2) =M_{odd}(7^2)=2$.
\end{proof}

\begin{lemma}\label{step5}
Assume that $d=p^2q$, with $p$ and $q$ two distinct odd primes. 
Then $M_{odd}(d)\leq 2$ if and only if $d\in\{3^2\cdot 7=63, 5^2\cdot 7 =175, 3^2\cdot 23= 207\}$.
\end{lemma}

\begin{proof}
Assume that $M_{odd}(p^2q)\leq 2$.\\ 
{\bf First}, we prove that $p=3$ or $q\leq 23$.\\
Indeed, suppose that $p>3$ and $q>23$ and let us get a contradiction. 
Then, $p^2\leq (3p)^2\leq (5p)^2\leq p^2q=d$ 
and $d+p^2=p^2(q+1)$, $d+(3p)^2 =p^2(q+9)$ and $d+(5p)^2=p^2(q+25)$. 
Since $p$ divides at most one of the three even integers $q+1$, $q+9$ and $q+25$ 
(otherwise $p>3$ would divide $8$, $16$ or $24$), 
we get that at least two of the three integers $q+1$, $q+9$ and $q+25$ are perfect powers of $2$. 
First, $q+1=2^\alpha$ and $q+9=2^\beta$ would imply $8=2^\alpha(2^{\beta-\alpha}-1)$ 
and $\alpha=3$ and the contradiction $8=2^\alpha=q+1>23+1$. 
Second, $q+9=2^\alpha$ and $q+25 =2^\beta$ would imply $16=2^\alpha(2^{\beta-\alpha}-1)$ 
and $\alpha=4$ and the contradiction $16=2^\alpha=q+9>23+9$. 
Third, $q+1=2^\alpha$ and $q+25=2^\beta$ would imply $24=2^\alpha(2^{\beta-\alpha}-1)$ 
and $\alpha=3$ and the contradiction $8=2^\alpha=q+1>23+1$.\\ 
{\bf Second}, we deal successively with each $q\leq 23$, i.e. with $q\in\{3,5,7,11,13,17,19,23\}$. 
If $q=23$, then $\omega(d+p^2)=\omega(24p^2)\leq 2$ yields $p=3$ and $d=3^2\cdot 23 =207$.
Then, the following Table shows that we cannot have $q\in\{5,11,13,17,19\}$:
$$\begin{array}{|r|r|r|r|r|r|}
\hline
q&d+p^2=(q+1)p^2&\omega(d+p^2)\leq 2\text{ implies}&d=p^2q&x&\omega(d+x^2)\\
\hline
5&6p^2&p=3&45&5&\omega(70)=3\\
\hline
11&12p^2&p=3&99&9&\omega(180)=3\\
\hline
13&14p^2&p=7&637&1&\omega(638)=3\\
\hline
17&18p^2&p=3&154&1&\omega(154)=3\\
\hline
19&20p^2&p=5&475&1&\omega(476)=3\\
\hline
\end{array}$$
Finally, for $q\in\{3,7\}$ Lemma \ref{step1}$(i)$ applied with $q$ instead of $p$, 
we have $9q^2>d$, i.e. $p^2<9q$. 
For $q=7$ we get $p\leq 5$, hence $d=3^2\cdot 7=63$ or $d=5^2\cdot 7=175$. 
For $q=3$ we get $p=5$, $d=5^2\cdot 3 =75$, 
for which $M_{odd}(75)\geq\omega(75+3^2)=\omega(84) =3$.\\
{\bf Third,} we finally deal with the case $p=3$ and $d=9q$ with $q\geq 5$ prime. 
Since $\omega(d+9^2)=\omega(9(q+9))\leq 2$ and $\omega(d+81^2)=\omega(9(q+729))\leq 2$ 
would imply $q+9=2^\alpha$, $q+729=2^\beta$ 
and the impossibility $2^\alpha(2^{\beta-\alpha}-1)=729-9 =2^4\cdot 45$, 
we have $d<81^2$, i.e. $q =d/9<729$.
For $q\geq 11$ we have $9^2\leq d$ 
and we must have $\omega (d+9^2) =\omega(9(q+9))\leq 2$, 
i.e. $q\in\{2^\alpha-9;\ \alpha\geq 4\} =\{7,23,503,2039,\cdots\}$. 
Hence, we must have $q\in\{5,7,23,503\}$. 
Since
$M_{odd}(3^2\cdot 5)\geq\omega(45+5^2)=3$ 
and $M_{odd}(3^2\cdot 503)\geq\omega(3^2\cdot 503+3^2)=3$ 
we get $q\in\{7,23\}$ and the desired result follows.
\end{proof}

\section{Notations}
To prove Theorems \ref{mainthimpairnot7mod8}, \ref{mainthimpair7mod8} and \ref{mainthd=2mod4}, 
we will need some results on algebraic number theory of imaginary quadratic number fields. 
So we need to set some notation we will be using throughout the paper.\\
We let 
${\mathbb K} 
={\mathbb Q}(\sqrt{-d})
=\{x+y\sqrt{-d};\ x,y\in {\mathbb Q}\}$ denote an imaginary quadratic number field, 
where $d>0$ is a positive rational integer. 
For $\alpha =x+y\sqrt{-d}\in {\mathbb K}$,
its conjugate is $\alpha' =x-y\sqrt{-d}\in {\mathbb K}$. 
We let ${\mathbb Z}_{\mathbb K}$, $d_{\mathbb K}<0$, $h_{\mathbb K}$ and ${\mathcal Cl}_{\mathbb K}$ 
denote its ring of algebraic integers, its negative its discriminant, its class number and its ideal class group. 
Recall that genus theory gives that 
the $2$-rank of ${\mathcal Cl}_{\mathbb K}$ is equal to $\omega(d_{\mathbb K})-1$. 
Moreover, if $d\geq 1$ is square-free 
and $p_1,\cdots,p_r$ are the $r\geq 0$ odd prime divisors of $d_{\mathbb K}$ 
then the ideal class $[{\mathcal P}_l]\in {\mathcal Cl}_{\mathbb K}$ 
of a not inert prime ideal ${\mathcal P}_l$ of prime norm $\ell\geq 2$ 
is a square in ${\mathcal Cl}_{\mathbb K}$ if and only if the $r$ Legendre's symbolds 
$\left (\frac{\ell}{p_k}\right )$ are equal to $+1$ for $1\leq k\leq r$.\\
For $\alpha\in {\mathbb Z}_{\mathbb K}$, 
we let $(\alpha)=\alpha{\mathbb Z}_{\mathbb K}$ denote the principal ideal generated by $\alpha$. 
For ${\mathcal I}$ a nonzero ideal of ${\mathbb Z}_{\mathbb K}$, 
we let $[{\mathcal I}]$ denote its ideal class in ${\mathcal Cl}_{\mathbb K}$.
If ${\mathcal I}$ an ideal of ${\mathbb Z}_{\mathbb K}$ then 
its norm $N({\mathcal I})$ is the order of the additive quotient group ${\mathbb Z}_{\mathbb K}/{\mathcal I}$ 
and its conjugate ideal ${\mathcal I}'$ is the ideal 
${\mathcal I}' =\{\alpha';\ \alpha\in {\mathcal I}\}$. 
The product ${\mathcal I}{\mathcal I}'$ is equal to the principal ideal $(N({\mathcal I}))$. 
Consequently, $[{\mathcal I}'] =[{\mathcal I}]^{-1}$ in ${\mathcal Cl}_{\mathbb K}$.
An ideal ${\mathcal I}$ is primitive 
if $n\in {\mathbb Z}$ and $(n)$ divides ${\mathcal I}$ imply $n=\pm 1$. 
Any ideal class in the class group of ${\mathbb K}$ contains a primitive ideal.
A primitive ideal ${\mathcal I}$ is a free ${\mathbb Z}$-module of the form 
${\mathcal I} =Q{\mathbb Z}+\frac{P+\sqrt{d_{\mathbb K}}}{2}{\mathbb Z}$, 
where $Q=N({\mathcal I})$ is the norm of ${\mathcal I}$ and $4Q$ divides $d_{\mathbb K}-P^2$. 
We may assume that $-Q<P\leq Q$. 
The conjugate ideal ${\mathcal I}'$ of ${\mathcal I}$ is the ideal 
${\mathcal I}' =Q{\mathbb Z}+\frac{-P+\sqrt{d_{\mathbb K}}}{2}{\mathbb Z}$. 
Since the product ${\mathcal I}{\mathcal I}'$ is clearly included in the principal ideal $(Q)$ 
and since both these ideals have norm $Q^2$, they are equal. 
Hence, we recover that ${\mathcal I}'$ is in the inverse of the ideal class of ${\mathcal I}$, 
i.e. that $[{\mathcal I}'] =[{\mathcal I}]^{-1}$ in ${\mathcal Cl}_{\mathbb K}$.
Finally, we will use:

\begin{lemma}\label{lnotinert}
Let $d>0$ be square-free. 
If an odd prime $\ell\geq 3$ is not inert in ${\mathbb K}={\mathbb Q}(\sqrt{-d})$ 
then there exists an integer $x\in\{0,2,\cdots,\ell\}$ such that $\ell\mid d+x^2$. 
We can also require $x$ to be even or odd.
\end{lemma}

\begin{proof}
The Legendre symbol $\left (\frac{d_{\mathbb K}}{\ell}\right ) =\left (\frac{-d}{\ell}\right )$ 
is equal to $0$ or $+1$. 
Hence, $-d$ is a square modulo $\ell$, 
i.e. there exists $x\in {\mathbb Z}$ such that $\ell$ divides $d+x^2$. 
Reducing $x$ modulo $\ell$ we may assume that $0\leq x\leq\ell$.
Changing $x$ to $\ell-x$ if necessary, we can force $x$ to be odd or to be even. 
\end{proof}

We know that any ideal class in ${\mathcal Cl}_{\mathbb K}$ contains an ideal of norm less than or equal to 
the Minkowski's bound 
$$\left (\frac{4}{\pi}\right )^{r_2}\frac{n!}{n^n}\sqrt{\vert d_{\mathbb K}\vert} 
=\frac{2}{\pi}\sqrt{\vert d_{\mathbb K}\vert}.$$ 
We improve on the Minkowski's bound for imaginary quadratic number fields:

\begin{lemma}\label{MinkowskisqrtdK/3}
Let ${\mathbb K}$ be an imaginary quadratic number field. 
Then any ideal class in the ideal class group ${\mathcal Cl}_{\mathbb K}$ of ${\mathbb K}$ 
contains an ideal of norm $\leq\sqrt{\vert d_{\mathbb K}\vert/3}$.
In particular, ${\mathcal Cl}_{\mathbb K}$ is generated by the ideal classes of the non inert prime ideals 
of prime norms $\leq\sqrt{\vert d_{\mathbb K}\vert/3}$.
\end{lemma}

\begin{proof}
Let ${\mathcal I} =Q{\mathbb Z}+\frac{P+\sqrt{d_{\mathbb K}}}{2}{\mathbb Z}$ 
be a primitive ideal of smallest norm $Q$ in a given ideal class ${\mathcal C}$, 
where $-Q<P\leq Q$ and $4Q$ divides $d_{\mathbb K}-P^2$.
We set $Q'=(P^2-d_{\mathbb K})/(4Q)>0$ 
and define the ideal ${\mathcal J} =Q'{\mathbb Z}+\frac{P+\sqrt{d_{\mathbb K}}}{2}{\mathbb Z}$.
Clearly, the product ${\mathcal I}{\mathcal J}$ is included in the principal ideal $(P+\sqrt{d_{\mathbb K}})/2)$. 
Both these ideals having norm $QQ'$, they are equal. 
Hence, $[{\mathcal J}] =[{\mathcal I}]^{-1} =[{\mathcal I}']$.
Therefore, $Q'\geq Q$ and 
$4Q^2\leq 4QQ' =P^2-d_{\mathbb K}\leq Q^2-d_{\mathbb K}$. 
The desired result follows.
\end{proof}

\section{Proof of Theorem \ref{mainthimpairnot7mod8}}
To begin with, $d$ must be a prime. 
Indeed, assume that $d=pq$ is a product of two distinct odd primes $p<q$ 
and that $M_{odd}(d)\leq 2$. 
Then $\omega(d+p^2)=\omega(p(p+q))\leq M_{odd}(d)\leq 2$ gives $8\leq p+q=2^a$ for some $a\geq 3$, 
which implies $d+p^2 =p(p+q) \equiv 0\pmod{8}$ 
and hence $d\equiv -p^2\equiv 7\pmod{8}$.
We devote a subsection to each one of these three possible cases $d\equiv 1,\ 3,\ 5\pmod{8}$.

\subsection{The case $d\equiv 3\pmod 8$ is prime}
\noindent\frame{\vbox{
\begin{proposition}\label{cased=p=3mod8}
Let $p\equiv 3\pmod 8$ be a prime. 
Set ${\mathbb K}={\mathbb Q}(\sqrt{-p})$, an imaginary quadratic number field 
of discriminant $d_{\mathbb K} =-p$.
If $M_{odd}(p)\leq 2$, then $h_{\mathbb K} =1$. 
Therefore, $M_{odd}(p)\leq 2$ if and only if $p\in\{3, 11, 19, 43, 67, 163\}$.
\end{proposition}
}}

\begin{proof}
Assume that $M_{odd}(p)\leq 2$.
We prove that any prime $\ell\leq\sqrt{\vert d_{\mathbb K}\vert/3} =\sqrt{p/3}$ is inert in ${\mathbb K}$. 
Hence $h_{\mathbb K}=1$ by Lemma \ref{MinkowskisqrtdK/3} 
and \cite{Watkins} gives the second assertion.\\
Since $d_{\mathbb K} =-p\equiv 5\pmod 8$, the prime $q=2$ is inert in ${\mathbb K}$. 
Now, 
assume that some odd prime $\ell\leq\sqrt{p/3}$ is not inert in ${\mathbb K}$
and let us arrive at some contradiction. 
Then $p>3$ 
and there exists $x\in {\mathbb Z}$ odd with $1\leq x\leq\ell\leq\sqrt{p/3}$ such that $\ell$ divides $p+x^2$, 
by Lemma \ref{lnotinert}. 
Since $p+x^2\equiv 4\pmod 8$, $\ell\mid d+x^2$ and $M_{odd}(p)\leq 2$, 
we get 
$$\hbox{$p+x^2=4\ell^a$ for some $a\geq 1$ and some odd $x$ with $1\leq x\leq l\leq\sqrt{p/3}$.}$$ 
Since $\ell^a=(p+x^2)/4\geq (p+1)/4>\sqrt{p/3}\geq\ell$, we have $a\geq 2$. 
Consequently,
$$l\nmid x.$$
The positive odd rational integer 
$$x':=2\ell-x\geq\ell >0$$
satisfies 
$$1\leq x'\leq\sqrt{p}.$$
Indeed, $x'=f(\ell)$, 
where $f(t):=2t-\sqrt{4t^a-p}$ is such that 
$f'(t)=2(1-at^{a-1}/\sqrt{4t^a-p})$ $\leq 0$ for $t\geq 1$
(notice that $a\geq 2$ yields $4t^a-p\leq 4t^a\leq a^2t^{2a-2}$ for $t\geq 1$). 
Hence, letting $t_0\geq 1$ denote the positive real number defined by $p+1=4t_0^a$, 
we do get 
$x'=f(\ell)\leq f(t_0) =2t_0-1\leq\sqrt{4t_0^a}-1=\sqrt{p+1}-1\leq\sqrt{p}$.\\ 
Moreover, $d+x'^2=d+x^2-4\ell x+4\ell^2 =4\ell^a+4\ell^2-4\ell x\equiv -4\ell x\pmod{\ell^2}$. 
Hence, $1\leq x'\leq\sqrt{p}$, $x'$ is odd, $p+x'^2\equiv 2\pmod 4$, $\ell\mid d+x'^2$ and $\ell^2\nmid d+x'^2$. 
Since $M_{odd}(p)\leq 2$ we obtain 
$$p+x'^2=4l$$ 
and the desired contradiction 
$\ell =(p+x'^2)/4\geq (p+1)/4>\sqrt{p/3}\geq\ell$.
\end{proof}

\subsection{The case $d\equiv 5\pmod 8$ is prime}
We need to improve upon Lemma \ref{MinkowskisqrtdK/3} 
that gives a Minkowski bound $\sqrt{\vert d_{\mathbb K}\vert/3} =\sqrt{4d/3}>\sqrt{d}$ 
to get a Minkowski bound $\leq\sqrt{d}$.
We give a more precise statement and different proof of \cite[Lemma 2.1]{GicaIndMath}:

\begin{lemma}\label{Minkowski}
Let $d\equiv 1,\ 2\pmod 4$ be a positive square-free integer. 
Set ${\mathbb K} ={\mathbb Q}(\sqrt{-d})$, 
an imaginary quadratic number field 
of ring of algebraic integers ${\mathbb Z}_{\mathbb K} ={\mathbb Z}[\sqrt{-d}]$ 
and discriminant $d_{\mathbb K} =-4d$. 
Let ${\mathcal Q}_2$ denote the prime ramified ideal above $2$, 
i.e. $(2)={\mathcal Q}_2^2$ in ${\mathbb Z}[\sqrt{-d}]$. 
Then any ideal class in the ideal class group ${\mathcal Cl}_{\mathbb K}$
contains an ideal of the form ${\mathcal Q}_2^a{\mathcal I}$, 
where ${\mathcal I}$ is an ideal of odd norm $\leq\sqrt{d}$. 
In particular, ${\mathcal Cl}_{\mathbb K}$ is generated by the ideal classes of the non inert prime ideals 
${\mathcal Q}$ of prime norm $q=N({\mathcal Q})\leq\max(2,\sqrt{d})$.
\end{lemma}

\begin{proof}
Now, let ${\mathcal L}$ be a primitive ideal in a given ideal class ${\mathcal C}$. 
Write ${\mathcal L}={\mathcal Q}_2^a{\mathcal I}$, 
where $a\in\{0,1\}$ 
and ${\mathcal I} =Q{\mathbb Z}+(P+\sqrt{-d}){\mathbb Z}$ is a primitive ideal of odd norm $Q$, 
where $Q$ divides $d+P^2$ and $0\leq P\leq Q$. 
We take ${\mathcal L}\in {\mathcal C}$ with $Q$ as small as possible 
and prove that $Q\leq\sqrt{d}$.
Since $Q$ is odd, $P$ or $P'=Q-P$ is of the same parity as $d$.\\
First, assume that $P$ is of the same parity as $d$. 
Then 
${\mathcal Q}_2 =2{\mathbb Z}+(P+\sqrt{-d}){\mathbb Z}$ 
and $d+P^2\equiv 2\pmod 4$.
We set ${\mathcal J} =Q'{\mathbb Z}+(P+\sqrt{-d}){\mathbb Z}$, 
where $Q' =(d+P^2)/(2Q)$ is odd. 
Then 
${\mathcal Q}_2{\mathcal I}{\mathcal J}
=2QQ'{\mathbb Z}+(P+\sqrt d){\mathbb Z} 
=(d+P^2){\mathbb Z}+(P+\sqrt d){\mathbb Z}$ 
$=(P+\sqrt d)$ is principal. 
Therefore, 
in ${\mathcal Cl}_{\mathbb K}$ we have 
${\mathcal C} 
=[{\mathcal Q}_2{\mathcal I}]
=[{\mathcal J}]^{-1}
=[{\mathcal J}']$ if $a=1$ 
and 
${\mathcal C} 
=[{\mathcal I}]
=[{\mathcal Q}_2{\mathcal J}]^{-1}
=[{\mathcal Q}_2{\mathcal J}']$ if $a=0$. 
By the minimality of $Q$ it follows that $Q'\geq Q$, 
i.e. that 
$d+P^2\geq 2Q^2$, 
which in using $0\leq P\leq Q$ gives $Q\leq\sqrt d$, as desired.\\
Second, assume that $P'$ is of the same parity as $d$. 
The same reasoning with $P'$ replacing $P$ and the conjugate ideal ${\mathcal J}'$ replacing ${\mathcal J}$ 
gives the same conclusion $Q\leq\sqrt d$, as desired.
\end{proof}

\noindent\frame{\vbox{
\begin{proposition}\label{casedsquarefree=5mod8}
Let $d\geq 5$ be square-free and such that $d\equiv 5\pmod 8$ be a prime. 
Set ${\mathbb K}={\mathbb Q}(\sqrt{-d})$, 
an imaginary quadratic number field 
of ring of algebraic integers ${\mathbb Z}_{\mathbb K} ={\mathbb Z}[\sqrt{-d}]$ 
and discriminant $d_{\mathbb K} =-4d$.
If $M_{odd}(d)\leq 2$, then $h_{\mathbb K} =2$. 
Therefore, $M_{odd}(d)\leq 2$ if and only if $d\in\{5, 13 ,37\}$.
\end{proposition}
}}

\begin{proof}
Assume that $M_{odd}(d)\leq 2$.
We claim that any odd prime $\ell\leq\sqrt{d}$ is inert in ${\mathbb K}$. 
Hence, ${\mathcal Cl}_{\mathbb K}$ is generated by the ideal class of order $\leq 2$ 
of the prime ramified ideal above $2$, by Lemma \ref{Minkowski}.
The first assertion follows and \cite{Watkins} implies the second assertion.\\
Indeed, 
assume that some odd prime $\ell\leq\sqrt{d}$ is not inert in ${\mathbb K}$
and let us arrive at some contradiction. 
Then $d>3$ and there exists $x\in {\mathbb Z}$ odd with $1\leq x\leq\ell\leq\sqrt{d}$ 
such that $\ell$ divides $p+x^2$, 
by Lemma \ref{lnotinert}.
Since $d+x^2\equiv 2\pmod 4$, $\ell\mid d+x^2$ and $M_{odd}(d)\leq 2$, 
we have 
$$\hbox{$d+x^2=2\ell^a$ for some $a\geq 1$ and some odd $x$ with $1\leq x\leq\sqrt{d}$.}$$
Since $2\ell^a =d+x^2\geq d+1>2\sqrt{d}\geq 2l$ we have $a\geq 2$. 
Consequently, as $d$ is square-free we have $$\ell\nmid x.$$
Since $d+x^2\equiv 6\pmod 8$ and $2\ell^2\equiv 2\pmod 8$, 
we have $a\neq 2$. 
Hence, $a\geq 3$. 
The positive odd rational integer
$$x':=2\ell-x\geq\ell >0$$ 
satisfies 
$$1\leq x'\leq\sqrt{p}.$$
Indeed, $x'=f(\ell)$, 
where $f(t):=2t-\sqrt{2t^a-d}$ is such that 
$f'(t)=2-at^{a-1}/\sqrt{2t^a-d}$ $\leq 0$ for $t\geq 1$
(notice that $a\geq 2$ yields $2t^a-d\leq 2t^a\leq a^2t^{2a-2}$ for $t\geq 1$). 
Hence, letting $t_0\geq 1$ denote the positive real number defined by $d+1=2t_0^a$ 
we do get 
$x'=f(\ell)\leq f(t_0) =2t_0-1 =2\left (\frac{d+1}{2}\right )^{1/a}-1\leq 2\left (\frac{d+1}{2}\right )^{1/3}-1\leq\sqrt{d}$.\\
Moreover, $d+x'^2=d+x^2-4\ell x+4\ell^2 =2\ell^a+4\ell^2-4\ell x\equiv -4\ell x\pmod{2\ell^2}$. 
Hence, $1\leq x'\leq\sqrt{d}$, $x'$ is odd, $d+x'^2\equiv 2\pmod{4}$, $\ell\mid d+x'^2$ and $\ell^2\nmid d+x'^2$. 
Since $M_{odd}(d)\leq 2$ we obtain 
$$d+x'^2=2l$$ 
and the desired contradiction 
$\ell =(p+x'^2)/2\geq (d+1)/2>\sqrt{d}\geq\ell$.
\end{proof}

\subsection{The case $d\equiv 1\pmod 8$ is prime}
\begin{lemma}\label{Lemmacased=p=1mod8}
Let $p\equiv 1\pmod 8$ be a prime with $M_{odd}(p)\leq 2$.
Let $\ell\leq\sqrt{p}$ be an odd non inert prime in ${\mathbb K} ={\mathbb Q}(\sqrt{-p})$. 
Then there exists 
$x\in {\mathbb Z}$ such that $p+x^2=2\ell^2$. 
\end{lemma}

\begin{proof}
There exists $x\in {\mathbb Z}$ odd with $1\leq x\leq\ell\leq\sqrt{p}$ such that $\ell$ divides $p+x^2$, 
by Lemma \ref{lnotinert}. 
Since $p+x^2\equiv 2\pmod4$, $\ell$ divides $p+x^2$ and $\omega(p+x^2)\leq M_{odd}(p)\leq 2$ 
we have 
$$\hbox{$p+x^2=2\ell^a$ for some $a\geq 1$.}$$
Since $2\ell^a =p+x^2\geq p+1>2\sqrt p\geq 2\ell$, we have $a\geq 2$.
We claim that $a=2$, i. e. that 
$$p+x^2=2\ell^2.$$
Indeed, assume that $a\geq 3$ and let us arrive at some contradiction. 
Since $x$ is odd, $p\equiv 1\pmod 8$ and $x^2<p$, 
we have $x^2\leq p-8$ 
and $2(p-4)\geq p+x^2=2\ell^a\geq 2\ell^3$. 
Hence, $\ell\leq (p-4)^{1/3}$.
If follows that the odd positive integer $$x':=2\ell-x$$ satisfies $x'^2<p$ 
(since $1\leq x\leq \ell$ then $0\leq x'=2\ell-x\leq 2\ell-1\leq 2(p-4)^{1/3}-1<\sqrt{p}$ 
as $f(t)=\sqrt{t}-(2(t-4)^{1/3}-1)\geq 0$ for $t\geq 4$). 
Hence $\omega(p+x'^2)\leq 2$ and $p+x'^2=2\ell^b$ for some $b\geq 1$. 
Therefore, we have 
$2(\ell^b-\ell^a) 
=(p+x'^2)-(p+x^2)
=(2\ell-x)^2-x^2 
=4\ell(\ell-x).$ 
Since $\ell\neq p$, we have $\ell\nmid x$. 
Finally, $\ell\mid \ell^b-\ell^a$ and $\ell^2\nmid \ell^b-\ell^a=\ell(\ell-x)$ 
give $b=1$ and the contradiction $p+1\leq p+x'^2=2\ell\leq 2\sqrt{p}$.
\end{proof}

We can now give a neat proof of the second assertion of \cite[Theorem 1.1]{GicaIndMath}:

\noindent\frame{\vbox{
\begin{proposition}\label{cased=p=1mod8}
Let $p\equiv 1\pmod 8$ be a prime. 
Set ${\mathbb K}={\mathbb Q}(\sqrt{-p})$, an imaginary quadratic number field 
of discriminant $d_{\mathbb K} =-4p$ and ring of algebraic integers ${\mathbb Z}[\sqrt{-p}]$.
If $M_{odd}(p)\leq 2$, then $h_{\mathbb K} =4$ and ${\mathcal Cl}_{\mathbb K}$ is cyclic. 
Therefore, $M_{odd}(p)\leq 2$ if and only if $p\in\{17,73,97,193\}$.
\end{proposition}
}}

\begin{proof}
The $2$-rank of ${\mathcal Cl}_{\mathbb K}$ is equal to $\omega(d_{\mathbb K})-1=1$. 
Hence, its $2$-Sylow subgroup is cyclic 
and the order $h_4$ of the subgroup $H_4$ of ${\mathcal Cl}_{\mathbb K}$ 
of its ideal classes of order dividing $4$. 
Now, $(2)={\mathcal Q}_2^2$ is totally ramified in ${\mathbb K}$. 
Since $p>2$, the equation $x^2+py^2=2$ has no solution in rational integers 
and ${\mathcal Q}_2$ is not principal.
Hence, its ideal class $[{\mathcal Q}_2]$ is of order $2$ in ${\mathcal Cl}_{\mathbb K}$. 
Since the Legendre's symbol $\left (\frac{2}{p}\right)$ is equal to $+1$, 
by genus theory $[{\mathcal Q}_2]$ is a square in ${\mathcal Cl}_{\mathbb K}$. 
Therefore, $h_4=4$.\\
Now, assume that $M_{odd}(p)\leq 2$. 
Let $\ell\geq 3$ be a prime non-inert in ${\mathbb K}$ such that $\ell\leq\sqrt{p}$. 
We claim that the ideal classes of the prime ideals above $\ell$ are in $H_4$, 
which implies the desired result by Lemma \ref{Minkowski}.
Indeed, there exists $x\in {\mathbb Z}$ 
such that $p+x^2 =2\ell^2$, 
by Lemma \ref{Lemmacased=p=1mod8}.
Therefore, the principal and primitive ideal $(x+\sqrt{-p})$ 
is a product ${\mathcal Q}_2{\mathcal Q}_\ell^2$, 
where ${\mathcal Q}_\ell$ is one of the two prime ideals of ${\mathbb K}$ above $\ell$. 
Hence, ${\mathcal Q}_\ell^4$ is principal 
and the ideal classes of the prime ideals above $\ell$ are in $H_4$.\\
The last assertion follows from \cite{Watkins}.
\end{proof}

\section{Proof of Theorem \ref{mainthimpair7mod8}}
Theorem \ref{mainthimpair7mod8} follows from the following Theorems \ref{d=pq=7mod8} and \ref{d=p=7mod8} 
and from the Siegel-Tatuzwa theorem, e.g. see \cite{{Tatuzawa}}, \cite{LouCollMath108} 
and \cite[3. Proof of Theorem 3]{LouJNT129}. 
Notice that A. Gica proved that in this case the Ono invariant of ${\mathbb K}$ is equal to $h_{\mathbb K}$ 
and referred to \cite{LouJNT129} to settle this case.

\begin{lemma}\label{stepdequiv7pmod8}
Let $d\equiv 7\pmod 8$ be a square-free positive integer. 
Set ${\mathbb K} ={\mathbb Q}(\sqrt{-d})$, 
an imaginary quadratic number field of discriminant $d_{\mathbb K} =-d\equiv 1\pmod 8$ 
in which $2$ splits completely into a product of two distinct prime ideals. 
Let ${\mathcal Cl}_{\mathbb K}^{(2)}$ denote the cyclic subgroup of ${\mathcal Cl}_{\mathbb K}$ 
generated by any one of these two prime ideals.
Assume that $M_{odd}(d)\leq 2$. 
If $\ell\leq\sqrt{d}$ is an odd prime not inert in ${\mathbb K} ={\mathbb Q}(\sqrt{-d})$ 
then the ideal classes of the prime ideals of ${\mathbb K}$ above $\ell$ 
are in ${\mathcal Cl}_{\mathbb K}^{(2)}$. 
By Lemma \ref{MinkowskisqrtdK/3} it follows that ${\mathcal Cl}_{\mathbb K} ={\mathcal Cl}_{\mathbb K}^{(2)}$.
\end{lemma}

\begin{proof}
There exists $x\in {\mathbb Z}$ odd with $1\leq x\leq\ell\leq\sqrt{d}$ 
such that $\ell$ divides $d+x^2$, 
by Lemma \ref{lnotinert}.
Since $d+x^2\equiv 0\pmod 8$, $\ell\mid d+x^2$ and $\omega(d+x^2)\leq M_{odd}(d)\leq 2$ we have 
$$\hbox{$d+x^2=2^a\ell^b$ for some $a\geq 3$ and $b\geq 1$ and some odd $x$ with $1\leq x\leq\sqrt{d}$.}$$
Suppose that $b\geq 2$. 
Then $\ell\nmid x$ (as $d$ is square-free).
Then 
$$x':=2\ell-x\geq\ell\geq 1$$ 
is odd and 
$x'^2\leq 4\ell^2\leq\frac{1}{2}2^a\ell^b =\frac{d+x^2}{2}\leq\frac{d+\ell^2}{2}\leq d$. 
Moreover, 
$d+x'^2=d+x^2+4\ell^2-4\ell x =2^a\ell^b+4\ell^2-4\ell x\equiv -4\ell x\pmod{4\ell^2}$. 
Recalling that $\ell\nmid x$, 
we obtain , $8\mid d+x'^2$, $\ell\mid d+x'^2$ and $\ell^2\nmid d+x'^2$. 
Since $M_{odd}(d)\leq 2$, we get $d+x'^2=2^{a'}\ell$ for some $a'\geq 3$. 
Consequently, whether $b=1$ or $b\geq 2$ there exits $x''\in {\mathbb Z}$ odd and $a''\geq 3$ 
such that $d+x''^2=2^{a''}\ell$, 
i.e. such that the algebraic integer $\alpha:=(x''+\sqrt{-d})/2\in {\mathbb K}$ is of norm $2^{a''-2}\ell$. 
The principal ideal $(\alpha)$ being primitive, we have $(\alpha) ={\mathcal Q}_2^{a-2}{\mathcal Q}_\ell$ 
where the prime ideal ${\mathcal Q}_2$ is above $2$ and the prime ideal ${\mathcal Q}_\ell$ is above $\ell$. 
Hence, in ${\mathcal Cl}_{\mathbb K}$ 
we have $[{\mathcal Q}_\ell] =[{\mathcal Q}_2]^{-(a-2)} =[{\mathcal Q}_2']^{a-2}$. 
The desired result follows.
\end{proof}

\noindent\frame{\vbox{
\begin{proposition}\label{d=pq=7mod8}
Let $d=pq\equiv 7\pmod 8$ be a product of two distinct odd primes. 
Set ${\mathbb K}={\mathbb Q}(\sqrt{-d})$, 
an imaginary quadratic field of discriminant $d_{\mathbb K}=-d\equiv 1\pmod 8$ 
in which the prime $2$ 
splits into a product of two distinct prime ideals of norm $2$.
Assume that $M_{odd}(d)\leq 2$. 
Then ${\mathcal Cl}_{\mathbb K}$ is cyclic 
and generated by the ideal class of any of the two prime ideals above the prime $2$. 
Moreover, $h_{\mathbb K}$ is not large, namely 
$$h_{\mathbb K}
\leq\frac{\log\vert d_{\mathbb K}\vert}{\log 2}-2.$$
\end{proposition}
}}

\begin{proof}
We may assume that $3\leq p<q$. 
Hence $p\leq\sqrt{d}$ and $\omega(p(q+p))=\omega(d+p^2)\leq M_{odd}(d)\leq 2$, 
which implies
$$\hbox{$8\leq q+p=2^a$ for some $a\geq 3$.}$$
Therefore, 
$d+((q-p)/2)^2 =((q+p)/2)^2 =2^{2a-2}$ 
and $\alpha =((q+p)/2+\sqrt{-d})/2$ is an algebraic integer of ${\mathbb K}$ of norm $2^{2a-2}$.
Therefore, the primitive principal ideal $(\alpha) ={\mathcal Q}_2^{2a-2}$ 
is a perfect power of one of the two prime ideals of ${\mathbb K}$ above the spliting prime $2$. 
Hence $[{\mathcal Q}_2]$ is of order dividing $2a-2$. 
By Lemma \ref{stepdequiv7pmod8} it follows that $h_{\mathbb K}$ divides $2a-2$. 
Noticing that $2^{2a-4} =((p+q)/4)^2\leq pq/4=d/4=\vert d_{\mathbb K}\vert/4$, 
we get the desired bound on $h_{\mathbb K}$.
\end{proof}

For the case $d=p\equiv 7\pmod 8$ we will need the following result:

\begin{lemma}\label{existence}
Let $p\equiv\pm 1\pmod 8$ be a prime. 
There exists a pair of integers $(x,y)\in [0,\sqrt{p})\times [0,\sqrt{p})$ such that $p+x^2=2y^2$. 
\end{lemma}

\begin{proof}
The ring of algebraic integers of the real quadratic number field ${\mathbb Q}(\sqrt{2})$ 
is ${\mathbb Z}[\sqrt{2}]$. 
For $\alpha =x+y\sqrt{2}\in {\mathbb Z}[\sqrt{2}]$, its conjugate is $\alpha'=x-y\sqrt{2}$ 
and its norm is $N(\alpha) =\alpha\alpha'=x^2-2y^2$.
The ring ${\mathbb Z}[\sqrt{2}]$ is a principal ideal domain whose fundamental unit $\eta =1+\sqrt{2}$ 
is of norm $-1$. 
Since $p\equiv\pm 1\pmod 8$, the prime $p$ splits in ${\mathbb Z}[\sqrt{2}]$
as a product of two prime principal ideals of norm $2$, say $(p)={\mathcal PP'}$ 
with ${\mathcal P} =(\alpha)$, ${\mathcal P'} =(\alpha')$, 
$p=N({\mathcal P}) =\vert N(\alpha)\vert$ 
and $p=N({\mathcal P'}) =\vert N(\alpha')\vert$. 
By changing $\alpha$ into $\alpha\eta$ if necessary, 
we may assume that $N(\alpha)=-p$
Therefore, writing $\alpha =x+y\sqrt{2}$, we have $N(\alpha) =x^2-2y^2=-p$ 
and the desired result. 
\end{proof} 

We can now prove the fourth assertion of \cite[Theorem 1.1]{GicaIndMath}: 

\noindent\frame{\vbox{
\begin{proposition}\label{d=p=7mod8}
Let $p\equiv 7\pmod 8$ be prime. 
Set ${\mathbb K}={\mathbb Q}(\sqrt{-p})$, 
an imaginary quadratic field of discriminant $d_{\mathbb K}=-p\equiv 1\pmod 8$ 
in which the prime $(2)={\mathcal Q}_2{\mathcal Q}_2'$ 
splits into a product of two distinct prime ideals of norm $2$.
Assume that $M_{odd}(p)\leq 2$. 
Then ${\mathcal Cl}_{\mathbb K}$ is cyclic 
and generated by the ideal class of any of the two prime ideals above the prime $2$. 
Moreover, $h_{\mathbb K}$ is not large, namely 
\begin{equation}\label{boundhK}
h_{\mathbb K}
\leq\frac{\log\vert d_{\mathbb K}\vert}{\log 2}+1
\end{equation}
\end{proposition}
}}

\begin{proof}
There exist positive integers $x,y\in [0,\sqrt p]$ such that $p+x^2=2y^2$, 
by Lemma \ref{existence}. 
Since $p\equiv 7\pmod 8$, $x$ is odd, $y=2z$ is even 
and 
\begin{equation}\label{pxz}
p+x^2=8z^2.
\end{equation}
Now, there are two cases.\\
$(i)$. Suppose that $z=2^u$ is a perfect power of $2$. 
Then $$p+x^2=2^{3+2u}$$
and $\alpha =(x+\sqrt{-p})/2$ is an algebraic integer of ${\mathbb K}$ of norm $2^{2u+1}$.
Therefore, the primitive principal ideal $(\alpha) ={\mathcal Q}_2^{2u+1}$ 
is a perfect power of one of the two prime ideals of ${\mathbb K}$ above the spliting prime $2$. 
Hence $[{\mathcal Q}_2]$ is of order dividing $2u+1$. 
By Lemma \ref{stepdequiv7pmod8} it follows that $h_{\mathbb K}$ divides $2u+1$. 
Consequently, since $2^{2+2u} =(p+x^2)/2\leq p$, we get
$h_{\mathbb K}\leq 1+2u\leq \frac{\log p}{\log 2}-1$ and \eqref{boundhK} holds true.\\
$(ii)$. Suppose that $z$ is not a perfect power of $2$. 
We use A. Gica's very nice trick.
Then $2x^2\leq p+x^2=8z^2$ and $8z^2=p+x^2\leq 2p$ yield $0\leq x\leq 2z\leq\sqrt p$. 
Hence, $0\leq 2z-x\leq\sqrt{p}$, $3z-x\geq 0$ and by \eqref{pxz}, we have
$$p+(2z-x)^2
=4z(3z-x).$$ 
Clearly, $\gcd(z,3z-x)=\gcd(z,x)=1$, by \eqref{pxz}.
Since $\omega(p+(2z-x)^2)\leq M_{odd}(p)\leq 2$ and $z$ is not a power of $2$, 
it follows that $3z-x=2^u$ is a perfect power of $2$. 
Now, $p+(-3x+8z)^2=8z^2-x^2+(-3x+8z)^2=8(3z-x)^2=2^{3+2u}$, 
i.e. 
$$p+(-3x+8z)^2 =2^{3+2u}.$$ 
As in case $(i)$ we obtain that $h_{\mathbb K}$ divides $2u+1$.
Moreover, $-3x+8z =\sqrt{8p+8x^2}-3x$, by \eqref{pxz}. 
Since $x\mapsto \sqrt{8p+8x^2}-3x$ is decreasing in the range $x\geq 0$, 
we get $0\leq -3x+8z\leq\sqrt{8p}$. 
Hence, 
$2^{1+2u} =(p+(-3x+8z)^2)/4\leq 9p/4$ 
and 
$h_{\mathbb K}\leq 1+2u\leq \frac{\log p}{\log 2}+1$ and \eqref{boundhK} holds true.
\end{proof}

\section{Proof of Theorem \ref{mainthd=2p}}
\begin{lemma}\label{d=2mod4}
Let $d> 2$ with $d\equiv 2\pmod 4$ be a square-free integer. 
Assume that $M_{even}(d)\leq 2$. 
Let $\ell\leq\sqrt{d}$ be an odd prime that splits in ${\mathbb K} ={\mathbb Q}(\sqrt{-d})$. 
Then, there exists $x\in {\mathbb Z}$
such that $d+x^2=2\ell^2$. 
\end{lemma}

\begin{proof}
By Lemma \ref{lnotinert}, there exists $x\in {\mathbb Z}$ even with $0\leq x\leq \ell\leq\sqrt{d}$ 
such that $\ell$ divides $d+x^2$. 
Since $l$ splits we have $x\neq 0$, 
i.e. $2\leq x\leq \ell\leq\sqrt{d}$.
Since, $2$ and $\ell$ divide $d+x^2$ and $\omega(d+x^2)\leq M_{even}(d) \leq 2$ 
we get 
$$\hbox{$d+x^2 =2\ell^a$ for some $a\geq 1$.}$$
Now, $d<d+x^2=2\ell^a\leq 2d^{a/2}$ gives $a\geq 2$. 
In particular, $d$ being square-free we have $\ell\nmid x$.
We claim that $a=2$, which proves the desired result. 
Indeed, suppose that $a\geq 3$ and let us arrive at a contradiction. 
Then the even positive rational integer 
$$x':=2\ell-x\geq\ell\geq 2$$ 
satisfies
$$x'\leq\sqrt{d}.$$
Indeed, we have 
$x'=2\ell-x=f(\ell)$, 
where $f(t) =2t-\sqrt{2t^a-d}$ is such that $f'(t) =(2\sqrt{2t^a-d}-at^{a-1})/\sqrt{2t^a-\ell}\leq 0$ for $t\geq 1$ 
(notice that $a\geq 3$ yields $4(2t^a-\ell)\leq 8t^a\leq a^2t^{2a-2}$ for $t\geq 1$). 
Hence, letting $t_0\geq 1$ denote the positive real number defined by $d+4=2t_0^a$, 
we get $x'=f(\ell)\leq f(t_0) =2\left (\frac{d+4}{2}\right )^{1/a}-2\leq 2\left (\frac{d+4}{2}\right )^{1/3}-2\leq\sqrt{d}$, 
as $(\sqrt{d}+2)^3-4(d+4) =d\sqrt{d}+2d+12\sqrt{d}-8\geq 0$).\\
Since $d+x'^2\equiv 2\pmod 4$, $\ell\mid d+x'^2$ and $\omega(d+x'^2)\leq M_{even}(d)\leq 2$, 
we would have $d+x'^2=2\ell^b$ for some $b\geq 1$. 
Since $\ell\nmid x$
and $2\ell^b=d+x'^2=2\ell^a-x^2+(2\ell-x)^2 =2\ell^a-4\ell x+4\ell^2\equiv -4\ell x\pmod{\ell^2}$
we get $b=1$ and the contradiction 
$(d+4)^3\leq (d+x'^2)^3 =8\ell^3\leq 8\ell^a =4(d+x^2)\leq 8d$, 
as $(d+4)^3-8d=d^3+12d^2+40d+64>0$. 
\end{proof}

Now, if $d\equiv 2\pmod 4$ is as in Theorem \ref{mainthd=2p}, 
then $\omega(d)\leq 2$. 
Hence either $d=2$ or $d=2p$ for some odd prime $p\geq 3$. 
Hence Theorem \ref{mainthd=2p} follows from Proposition \ref{propd=2p}:

\begin{proposition}\label{propd=2p}:
Let $p\geq 3$ be an odd prime. 
Set $d=2p$ and ${\mathbb K} ={\mathbb Q}(\sqrt{-d})$, 
an imaginary quadratic field of discriminant $d_{\mathbb K}=-4d$ 
and ring of algebraic integers ${\mathbb Z}_{\mathbb K} ={\mathbb Z}[\sqrt{-d}]$.
Assume that $M_{even}(d)\leq 2$.\\
$(i)$. If $p\equiv\pm 3\pmod 8$ then any prime $\ell\leq\sqrt{d}$ is inert or ramified in ${\mathbb K}$ 
and $h_{\mathbb K}=2$.\\ 
$(ii)$ If $p\equiv\pm 1\pmod 8$ then ${\mathcal Cl}_{\mathbb K}$ is cyclic and $h_{\mathbb K}=4$.
\end{proposition}

\begin{proof}
Let $\ell$ be an odd prime that splits in ${\mathbb K}$. 
By Lemma \ref{d=2mod4} there exists $x\in {\mathbb Z}$ such that $d+x^2=2\ell^2$. 
It follows that the principal and primitive ideal $(x+\sqrt{-d})$ is a product ${\mathcal Q}_2{\mathcal Q}_l^2$, 
where ${\mathcal Q}_2$ is the not principal prime ramified ideal above $2$ 
and ${\mathcal Q}_l$ is one of the two prime ideals of ${\mathbb K}$ above $\ell$. 
Hence the ideal class $[{\mathcal Q}_2]$ is a square in ${\mathcal Cl}_{\mathbb K}$. 
By genus theory it follows that the Legendre symbol $\left (\frac{2}{p}\right )$ is equal to $+1$, 
which amounts to having $p\equiv\pm 1\mod 8$.
By Lemma \ref{Minkowski}, and since the ideal class $[{\mathcal Q}_2]$ is of order $2$, point $(i)$ follows. 
Now, if $p\equiv\pm 1\pmod 8$, then the Legendre symbol $\left (\frac{2}{p}\right )$ is equal to $+1$ 
and by genus theory the ideal class $[{\mathcal Q}_2]$ is a square in ${\mathcal Cl}_{\mathbb K}$. 
Since the $2$-rank of ${\mathcal Cl}_{\mathbb K}$ is $\omega (d_{\mathbb K})-1 =1$, 
its $2$-Sylow subgroup is cyclic and its subgroup $H_4$ of the ideal classes of order dividing $4$ 
is of order $4$. 
Since ${\mathcal Q}_2{\mathcal Q}_l^2$ is principal, the ideal class of ${\mathcal Q}_l$ is in $H_4$. 
Hence ${\mathcal Cl}_{\mathbb K} =H_4$ is cyclic of order $4$, by Lemma \ref{Minkowski}.
\end{proof}

\section{Proof of Theorem \ref{mainthd=2mod4}}
Let $d\equiv 2\pmod 4$ with $M_{even}(d)\leq 2$ be square-free.
Clearly, we may assume that $d>2$. Hence $\omega(d)\geq 2$.
We claim that $\omega(d)\leq 3$, i.e. that 
$d=2p$ 
or $d=2p_1p_2$. 
Indeed, if $\omega(d)\geq 4$ 
then $d=2pd'$ with $p\geq 3$ a prime and $d'>p^2>2p$ odd and not divisible by $p$. 
It follows that $2\leq 2p\leq\sqrt{d}$ and $M_{even}(d)\geq\omega(d+(2p)^2) =\omega(2p(d'+2p))\geq 3$.
Since Proposition \ref{propd=2p} deals with the case $d=2p$, 
we now deal with the case $d=2p_1p_2$.
There are two cases.

\noindent $(i)$. Assume that no odd prime $\ell\leq\sqrt{d}$ splits in ${\mathbb K}$. 
Then the ideal class group of ${\mathbb K}$ 
is generated by the ideal classes of the prime ramified ideals 
${\mathcal Q}_2$, ${\mathcal Q}_{p_1}$ and ${\mathcal Q}_{p_2}$ above $2$, $p_1$ and $p_2$, 
by Lemma \ref{Minkowski}. 
Now, the six equations $x^2+2p_1p_2y^2=N$ have clearly no solution in rational integers 
for $N\in\{2,p_1,p_2,2p_1,2p_2,p_2p_3\}$. 
Therefore, these classes are of order $2$, no product of two of them is the neutral element 
whereas the product of the three of them is the neutral element,
as the product ${\mathcal Q}_2{\mathcal Q}_{p_1}{\mathcal Q}_{p_2}$ is the principal ideal $(\sqrt{-d})$. 
Hence ${\mathcal Cl}_{\mathbb K}$ is isomorphic to ${\mathbb Z}/2{\mathbb Z}\times {\mathbb Z}/2{\mathbb Z}$ 
and $h_{\mathbb K} =4$.

\noindent $(ii)$. Otherwise, let $\ell\leq\sqrt{d}$ be an odd prime that splits in ${\mathbb K}$. 
Then $d+x^2
=2\ell^2$ for some $x\in {\mathbb Z}$, by Lemma \ref{d=2mod4}. 
It follows that the principal and primitive ideal $(x+\sqrt{-d})$ is a product ${\mathcal Q}_2{\mathcal Q}_\ell^2$, 
where ${\mathcal Q}_\ell$ is one of the two prime ideals of ${\mathbb K}$ above $\ell$. 
Since $d>2$, the ideal class $[{\mathcal Q}_2]$ is of order $2$. 
Hence, $[{\mathcal Q}_\ell]\in H_4$, 
where $H_4$ is the subgroup of ${\mathcal Cl}_{\mathbb K}$ 
of its ideal classes of order dividing $4$. 
By Lemma \ref{Minkowski} it follows that ${\mathcal Cl}_{\mathbb K}$ is a subgroup of $H_4$. 
Hence ${\mathcal Cl}_{\mathbb K} =H_4$. 
The $2$-rank of the ideal class group of ${\mathbb K} ={\mathbb Q}(\sqrt{-d})$ 
being equal to $\omega(d_{\mathbb K})-1 =2$, 
the order $h_{\mathbb K}$ of $H_4$ divides $4^{2} =16$.

\section{Proof of Theorem \ref{mainthd=2mod4notsquarefree}}
Theorem \ref{mainthd=2mod4notsquarefree} follows from the following two Lemmas.

\begin{lemma}
Let $d\geq 2$ with $d\equiv 2\pmod 4$ be a not square-free integer such that $M_{even}(d)\leq 2$. 
Then $d=2p^2$ for some prime $p\geq 3$. 
\end{lemma}

\begin{proof}
First, assume that $p^3$ divides $d =2p^3d'$ for some prime $p\geq 3$, with $d'\geq 1$ odd. 
Then $(2p)^2\leq d$. Hence, $M_{even}(d)\geq\omega(d+(2p)^2) =\omega(2p^2(pd'+2))\geq 3$.\\
Second, since $d$ is not square-free, take a prime $p\geq 3$ such that $p^2$ divides $d$. 
Write 
$$d=2p^2d'$$ 
with $d'$ odd and $p\nmid d'$, 
by the first point.\\
$(i)$. Assume that $d'\geq 9$. 
Then $(2p)^2\leq (4p)^2\leq 2p^2d'=d$. 
Hence $\omega(d+(2p)^2) =\omega(2p^2(d'+2))\leq M_{even}(d)\leq 2$ 
and $\omega(d+(4p)^2)=\omega(2p^2(d'+8)\leq M_{even}(d)\leq 2$. 
Therefore, $d'+2=p^a$ and $d'+8=p^b$ with $1\leq a<b$. 
It follows that $6=p^a(p^{b-a}-1)$, 
which implies the contradiction $9\leq d'=p^a-2\leq p^a(p^{b-a}-1)-2\leq 6-2=4$.\\
$(ii)$. Now assume that $3\leq d'\leq 7$. 
Then $(2p)^2\leq 2p^2d' =d$. 
Hence $\omega (ad+(2p)^2) =\omega(2p^2(d'+2))\leq M_{even}(d)\leq 2$. 
Therefore, $d'+2=p^a$ for some $a\geq 1$. 
It follows that either
$(a)$ $d'=3$, $p^a=5$ and $d=2\cdot 5^2\cdot 3 =150$ for which $M_{even}(150)=3$, 
or $(b)$ $d'=5$, $p^a=7$ and $d=2\cdot 7^2\cdot 5 =490$ for which $M_{even}(490)=3$, 
or $(c)$ $d'=7$, $p^a=9$ and $d=2\cdot 3^2\cdot 7 =126$ for which $M_{even}(126)=3$.\\
$(iii)$. Hence $d'=1$ and $d=2p^2$.
\end{proof}

\begin{lemma}
Let $d\geq 2$ be an integer of the form $d=2p^2$ for some prime $p\geq 3$. 
Then $M_{even}(d)\leq 2$ if and only if $d=2\cdot 3^2 =18$.
\end{lemma}

\begin{proof}
We argue as in Lemma \ref{step4}. 
Let $p\geq 3$ be an odd prime. 
Take $x_p\in [-16,17]\cap {\mathbb Z}$ such that $x_p\equiv 4p\pmod{33}$. 
Then $2x_p$ is even and $2p^2+(2x_p)^2\equiv 2p^2(1+2\cdot 4^2)\equiv 0\pmod{66}$. 
Hence $M_{even}(2p^2)\geq\omega(2p^2+(2x_p)^2)\geq 3$ for $34^2\leq 2p^2$, i.e. for $p\geq 25$. 
The computation of $M_{even}(2p^2)$ for $3\leq p<25$ yields the desired result.
\end{proof}

\section{Proof of Theorem \ref{mainthd=2mod4reel}}
\begin{lemma}\label{d=2mod4reel}
Let $d> 2$ with $d\equiv 2\pmod 4$ be a square-free integer. 
Assume that $M_{even}'(d)\leq 2$. 
Let $\ell\leq\sqrt{d}$ be an odd prime that splits in ${\mathbb K} ={\mathbb Q}(\sqrt{-d})$. 
Then, there exists $x$ even with $2\leq x\leq\sqrt{d}$ 
such that $d-x^2=2\ell$, that implies $x=l-1$ and $d=l^2+1$, or $d-x^2=2\ell^2$.
\end{lemma}

\begin{proof}
As in Lemma \ref{lnotinert}, 
there exists $x\in {\mathbb Z}$ even with $2\leq x\leq \ell\leq\sqrt{d}$ 
such that $\ell$ divides $d-x^2$. 
Since, $2$ and $\ell$ divide $d-x^2$ and $\omega(d-x^2)\leq M_{even}'(d) \leq 2$ 
we get 
$$\hbox{$d-x^2 =2\ell^a$ for some $a\geq 1$.}$$
Suppose that $a\geq 3$. 
Since $d$ is square-free, we have $\ell\nmid x$.
The even positive rational integer 
$$x':=2\ell-x\geq\ell\geq 2$$ 
satisfies
$$x'\leq 2l-2\leq 2((d-4)/2)^{1/a}-2\leq 2((d-4)/2)^{1/3}-2\leq\sqrt{d}.$$
Since $2$ and $\ell$ divide $d-x'^2$ and $\omega(d-x'^2)\leq M_{even}(d)'\leq 2$, 
we have $d-x'^2=2\ell^b$ for some $b\geq 1$. 
Since $\ell\nmid x$
and 
$2\ell^b 
=d-x'^2
=2\ell^a-x^2-(2\ell-x)^2 
=2\ell^a+4\ell x-4\ell^2
\equiv -4\ell x\pmod{\ell^2}$
we get $b=1$ 
and $d-x'^2=2\ell$. 
In that case, $x^2\leq l^2<d=x^2+2l\leq (l-1)^2+2l=l^2+1$
and hence $d=l^2+1$ and $x^2=d-2l=l^2+1-2l=(l-1)^2$.
The desired result follows.
\end{proof}

Now we can proceed with the proof of Theorem \ref{mainthd=2mod4} 
as we did for the proof of Theorem \ref{mainthd=2mod4}. 
The only difference being that in the real quadratic case 
the $2$-rank of the ideal class group is $\leq\omega(d_{\mathbb K})-1$ 
but not necessarily equal to $\omega(d_{\mathbb K})-1$.

\section{Proof of Theorem \ref{mainthd=2mod4notsquarefreereel}}
Theorem \ref{mainthd=2mod4notsquarefreereel} follows from the following two Lemmas.

\begin{lemma}
Let $d\geq 2$ with $d\equiv 2\pmod 4$ be a not square-free integer such that $M_{even}'(d)\leq 2$. 
Then $d=54$, for which $M_{even}'(54)=2$, 
$d =90$, for which $M_{even}'(90)=2$, 
$d=2p^2$ for some prime $p\geq 3$, 
or $d=6p^2$ for some prime $p\geq 3$. 
\end{lemma}

\begin{proof}
First, assume that $p^3$ divides $d =2p^3d'$ for some prime $p\geq 3$ and $d'\geq 1$ odd. 
Then $(2p)^2\leq d$. 
Hence, $\omega(d-(2p)^2)=\omega(2p^2(pd'-2))\leq M_{even}'(d)\leq 2$, 
that implies $pd'-2=1$. Thus, $p=3$ and $d'=1$, i.e. $d=2\cdot 3^3=54$.\\
Second, since $d$ is not square-free, take a prime $p\geq 3$ such that $p^2$ divides $d$. 
Write 
$$d=2p^2d'$$ 
with $d'$ odd and $p\nmid d'$, 
by the first point.\\
$(i)$. Assume that $d'\geq 9$. 
Then $(2p)^2\leq (4p)^2\leq 2p^2d' =d$. 
Hence $\omega(d-(2p)^2)=\omega(2p^2(d'-2))\leq M_{even}'(d)\leq 2$ 
and $\omega(d-(4p)^2) =\omega(2p^2(d'-8))\leq M_{even}'(d)\leq 2$. 
Therefore, $d'-2=p^a$ and $d'-8=p^b$ with $1\leq b<a$. 
It follows that $6=p^b(p^{a-b}-1)$, 
which implies $p^b=3$, $d'=p^b+8 =11$ and $ d =2p^2d' =198$. 
Moreover, $M_{even}'(198)=3$.\\
$(ii)$. Now assume that $3\leq d'\leq 7$, i.e. that $d'\in\{3,5,7\}$. 
Then $(2p)^2\leq 2p^2d' =d$. 
Hence $\omega(d-(2p)^2) =\omega(2p^2(d'-2))\leq M_{even}'(d)$. 
Therefore, $d'-2=p^a$ for some $a\geq 0$. 
It follows that either
$(a)$ $d'=3$ and $d=6p^2$, 
or $(b)$ $d'=5$, $p^a=3$ and $d=2\cdot 3^2\cdot 5 =90$ for which $M_{even}(90)=2$, 
or $(c)$ $d'=7$, $p^a=5$ and $d=2\cdot 5^2\cdot 7 =350$ for which $M_{even}(350)=3$.
\end{proof}

\begin{lemma}
Let $d\geq 2$ be an integer of the form $d=2p^2$ or $d=6p^2$ for some prime $p\geq 3$. 
Then $M_{even}'(d)\leq 2$ if and only if $d=18$, $50$ or $98$.
\end{lemma}

\begin{proof}
We argue as in Lemma \ref{step4}. 
Let $p\geq 3$ be an odd prime. 
First assume that $d=2p^2$.
Take $x_p\in [-80,81]\cap {\mathbb Z}$ such that $x_p\equiv 9p\pmod{161}$. 
Then $2x_p$ is even and $2p^2-(2x_p)^2\equiv 2p^2(1-2\cdot 9^2)\equiv 0\pmod{362}$. 
Hence $M_{even}(2p^2)\geq\omega(2p^2-(2x_p)^2)\geq 3$ for $162^2\leq 2p^2$, i.e. for $p\geq 115$. 
The computation of $M_{even}(2p^2)$ for $3\leq p<115$ yields the desired result. 
Second, assume that $d=6p^2$. 
Take $x_p\in [-7,8]\cap {\mathbb Z}$ such that $x_p\equiv 3p\pmod{15}$. 
Then $2x_p$ is even and $6p^2-(2x_p)^2\equiv 2p^2(3-2\cdot 3^2)\equiv 0\pmod{30}$. 
Hence $M_{even}(6p^2)\geq\omega(6p^2-(2x_p)^2)\geq 3$ for $16^2\leq 2p^2$, i.e. for $p\geq 12$. 
The computation of $M_{even}(6p^2)$ for $3\leq p<12$ yields the desired result. 
\end{proof}

\section{Concluding remarks}\label{remarks}
This article is a tribute to Gica's interesting paper \cite{GicaIndMath}. 
Our aim here has been to correct some inaccuracies in that paper, 
see below, 
and to begin exploring whether it is possible to transpose some of his results to the real quadratic case, 
see Theorem \ref{mainthd=2mod4reel}.

At the beginning of \cite[Section 6]{GicaIndMath} A. Gica proved that 
if $p\equiv 7\pmod 8$ is prime and $M_{odd}(p)\leq 2$, 
then ${\mathcal Cl}_{\mathbb K}$ is cyclic generated by the ideal class of of any of the two prime ideals of 
${\mathbb K}={\mathbb Q}(\sqrt{-p})$ above the splitting prime $2$. 
However, using \cite[Lemma 2.2]{GicaIndMath} to obtain this result 
he has to assume that $p>2916$. 
But according to \cite[Theorem 1]{GicaIndMath} all these fourteen values 
of $p\equiv 7\pmod 8$ with $M_{odd}(p)\leq 2$
(with at most one possible exception) are $\leq 1423$. 
Hence his proof says nothing on the structure of ${\mathcal Cl}_{\mathbb K}$ for these fourteen $p$'s.
Our present approach in Lemma \ref{stepdequiv7pmod8} is clearly more satisfactory.

In \cite[page 5, the case $p\equiv 1\pmod 8$]{GicaIndMath} 
A. Gica refers to \cite[Theorem 1]{GicaJNT} for the fact that for $p\equiv 1\pmod 8$ 
there exists a unique positive integer $x<\sqrt{p}$ such that $p+x^2=2y^2$ for some $y\in {\mathbb Z}$. 
But in \cite[Theorem 1]{GicaJNT} only the existence of such an $x$ that is proved, not its uniqueness. 
In Lemma \ref{existence} we also give a proof of this existence.
Our proof of Theorem \ref{cased=p=1mod8} does not require this uniqueness, 
which is in fact true according to a proof A. Gica sent us on September 2024. 

In \cite[Theorem 7.1]{GicaIndMath} the authors missed the case $d=135$ 
given in Theorem \ref{dnotprimenotpq}.

A. Gica wanted to prove Theorem \ref{mainthd=2p}, 
for which we have provided a different proof than the one he sent us on November 2025. 
Theorems \ref{mainthd=2mod4} and \ref{mainthd=2mod4notsquarefree} 
present complete proofs of results stronger than \cite[Theorem 7.2]{GicaIndMath}.

{\small

}
\end{document}